\documentclass[a4paper,12pt,twoside,leqno,final]{amsart}
\usepackage{amsmath}
\usepackage{amssymb}

\newtheorem{thm}{Theorem}[section]
\newtheorem{lem}[thm]{Lemma}

\newtheorem*{tha}{Theorem A}

\theoremstyle{definition}

\newcommand{\C}{{\mathbb C}}
\newcommand{\D}{{\mathbb D}}
\newcommand{\R}{{\mathbb R}}
\newcommand{\T}{{\mathbb T}}
\newcommand{\Z}{{\mathbb Z}}

\newcommand{\f}{\frac}
\newcommand{\ov}{\overline}

\newcommand{\al}{\alpha}
\newcommand{\be}{\beta}

\newcommand{\ga}{\gamma}

\newcommand{\ze}{\zeta}
\renewcommand{\th}{\theta}

\newcommand{\ph}{\varphi}

\newcommand{\const}{\text{\rm const}}

\numberwithin{equation}{section}

\title[A Rudin--de Leeuw type theorem]
{A Rudin--de Leeuw type theorem\\ 
for functions with spectral gaps}
\author{Konstantin M. Dyakonov}
\address{Departament de Matem\`atiques i Inform\`atica, IMUB, BGSMath, Universitat de Barcelona, Gran Via 585, E-08007 Barcelona, Spain}
\address{ICREA, Pg. Llu\'is Companys 23, E-08010 Barcelona, Spain}
\email{konstantin.dyakonov@icrea.cat}
\thanks{Supported in part by grant MTM2017-83499-P from El Ministerio de Econom\'ia y Competitividad (Spain) and grant 2017-SGR-358 from AGAUR (Generalitat de Catalunya).}

\begin{document}
\begin{abstract} 
Our starting point is a theorem of de Leeuw and Rudin that describes the extreme points of the unit ball in the Hardy space $H^1$. We extend this result to subspaces of $H^1$ formed by functions with smaller spectra. More precisely, given a finite set $\mathcal K$ of positive integers, we prove a Rudin--de Leeuw type theorem for the unit ball of $H^1_{\mathcal K}$, 
the space of functions $f\in H^1$ whose Fourier coefficients $\widehat f(k)$ vanish for all $k\in\mathcal K$. 
\end{abstract}

\maketitle

\section{Introduction and main result}

Let $\T$ stand for the unit circle $\{\ze\in\C:|\ze|=1\}$, endowed with normalized Lebesgue measure, and let $L^1=L^1(\T)$ be the space of all complex-valued integrable functions $f$ on $\T$, with norm 
\begin{equation}\label{eqn:lonenormcr}
\|f\|_1:=\f 1{2\pi}\int_\T|f(\ze)|\,|d\ze|.
\end{equation}
The {\it Fourier coefficients} of a function $f\in L^1$ are the numbers
$$\widehat f(k):=\f 1{2\pi}\int_\T\ov\ze^kf(\ze)\,|d\ze|,\qquad k\in\Z,$$
and the set 
$$\text{\rm spec}\,f:=\{k\in\Z:\,\widehat f(k)\ne0\}$$
is called the {\it spectrum} of $f$. 

\par Further, the {\it Hardy space} $H^1$ is defined by 
$$H^1:=\{f\in L^1:\,\text{\rm spec}\,f\subset\Z_+\}$$
and normed as above; here $\Z_+:=\{0,1,2,\dots\}$. The harmonic extension (given by the Poisson integral) of a function $f\in H^1$ to the disk $\D:=\{z\in\C:|z|<1\}$
is actually holomorphic there (see, e.g., \cite[Chapter II]{G}), so we may view elements of $H^1$ as holomorphic functions on $\D$. Recall also that a non-null function $F\in H^1$ is said to be {\it outer} if 
$$\log|F(0)|=\f 1{2\pi}\int_\T\log|F(\ze)|\,|d\ze|,$$
whereas a function $I$ of class $H^\infty:=H^1\cap L^\infty(\T)$ is termed {\it inner} if $|I|=1$ a.e. on $\T$. It is well known that a generic function $f\in H^1$, $f\not\equiv0$, admits an (essentially unique) factorization of the form 
\begin{equation}\label{eqn:fifcr}
f=IF,
\end{equation}
where $I$ is inner and $F$ is outer. We refer to any of \cite{G, Hof} or \cite{K} for basic facts about Hardy spaces, including the canonical factorization theorem just mentioned. 

\par This note is motivated by a beautiful theorem of de Leeuw and Rudin, which describes the extreme points of the unit ball in $H^1$. Before stating it, we need to introduce yet another piece of notation. Namely, given a Banach space $X=(X,\|\cdot\|)$, we write  
$$\text{\rm ball}(X):=\{x\in X:\,\|x\|\le1\}.$$
Finally, we recall that an element $x$ of $\text{\rm ball}(X)$ is said to be an {\it extreme point} thereof if it is not an interior point of any line segment contained in $\text{\rm ball}(X)$. Of course, any such point $x$ satisfies $\|x\|=1$. 

\par The Rudin--de Leeuw result that interests us here reads as follows. 

\begin{tha} A function $f\in H^1$ with $\|f\|_1=1$ is an extreme point of $\text{\rm ball}(H^1)$ if and only if it is outer. 
\end{tha}

\par The original proof can be found in \cite{dLR}; see also \cite[Chapter IV]{G} and \cite[Chapter 9]{Hof} for alternative presentations. 

\par We are concerned with certain finite-dimensional perturbations of Theorem A. Specifically, the question is what happens if $H^1$ gets replaced by a smaller subspace, whose elements are required to have some additional spectral holes (but not too many of them). To be more precise, we fix a finite number (say, $M$) of positive integers 
$$k_1<k_2<\dots<k_M$$
and restrict our attention to the functions $f\in H^1$ that satisfy 
$$\widehat f(k_1)=\dots=\widehat f(k_M)=0.$$
The subspace comprised of such functions is thus
$$H^1_{\mathcal K}:=\{f\in H^1:\,\text{\rm spec}\,f\subset
\Z_+\setminus\mathcal K\},$$
where 
\begin{equation}\label{eqn:defmathcalkcr}
\mathcal K:=\{k_1,\dots,k_M\}.
\end{equation}
Our purpose here is to characterize the extreme points of $\text{\rm ball}\left(H^1_{\mathcal K}\right)$, the unit ball of $H^1_{\mathcal K}$ endowed with the $L^1$ norm \eqref{eqn:lonenormcr}. 

\par Because $H^1_{\mathcal K}$ has finite codimension in $H^1$, one would not expect the situation to be very different from that in Theorem A. So, {\it a priori}, one feels that the extreme points $f$ of $\text{\rm ball}\left(H^1_{\mathcal K}\right)$ should probably be the unit-norm functions which are fairly close to being outer, in some sense or other. Our characterization, stated below in terms of the function's canonical factorization \eqref{eqn:fifcr}, justifies this guess and gives a precise meaning to the notion of a \lq\lq nearly outer" function that arises. 

\par First of all, it turns out that if $f\in\text{\rm ball}\left(H^1_{\mathcal K}\right)$ is extreme, then its inner factor, $I$, must be a {\it finite Blaschke product} whose degree (i.e., the number of its zeros) does not exceed 
$M(=\#\mathcal K)$. This means that $I$ is writable, possibly after multiplication by a unimodular constant, as
\begin{equation}\label{eqn:finblaprodegmcr}
I(z)=\prod_{j=1}^m\f{z-a_j}{1-\ov a_jz},
\end{equation}
where $0\le m\le M$ and $a_1,\dots,a_m$ are points in $\D$. (When $m=0$, it is of course understood that $I(z)=1$.) Secondly---and perhaps less predictably---there is an interplay between the two factors, $I$ and $F$, in \eqref{eqn:fifcr} which we now describe. 

\par Assuming that $I$ is given by \eqref{eqn:finblaprodegmcr} and $F\in H^1$ is outer, we consider the function 
\begin{equation}\label{eqn:deffzerocr}
F_0(z):=F(z)\prod_{j=1}^m(1-\ov a_jz)^{-2}
\end{equation}
and its coefficients 
$$C_k:=\widehat F_0(k),\qquad k\in\Z.$$
Since $a_1,\dots,a_m\in\D$, it follows that $F_0\in H^1$ and so $C_k=0$ for all $k<0$. Also, we define 
$$A(k):=\text{\rm Re}\,C_k,\qquad B(k):=\text{\rm Im}\,C_k\qquad(k\in\Z)$$
and introduce, for $j=1,\dots,M$ and $l=0,\dots,m$, the numbers 
$$A^+_{j,l}:=A(k_j+l-m)+A(k_j-l-m),\qquad B^+_{j,l}:=B(k_j+l-m)+B(k_j-l-m)$$
and
$$A^-_{j,l}:=A(k_j+l-m)-A(k_j-l-m),\qquad B^-_{j,l}:=B(k_j+l-m)-B(k_j-l-m).$$
(The integers $k_j$ are, of course, the same as in \eqref{eqn:defmathcalkcr}.) Next, we build the $M\times(m+1)$ matrices 
\begin{equation}\label{eqn:plusmatcr}
\mathcal A^+:=\left\{A^+_{j,l}\right\},\qquad\mathcal B^+:=\left\{B^+_{j,l}\right\}
\end{equation}
and the $M\times m$ matrices
\begin{equation}\label{eqn:minusmatcr}
\mathcal A^-:=\left\{A^-_{j,l}\right\},\qquad
\mathcal B^-:=\left\{B^-_{j,l}\right\}.
\end{equation}
Here, the row index $j$ always runs from $1$ to $M$. As to the column index $l$, it runs from $0$ to $m$ for each of the two matrices in \eqref{eqn:plusmatcr}, and from $1$ to $m$ for each of those in \eqref{eqn:minusmatcr}. 
\par Finally, we need the block matrix 
\begin{equation}\label{eqn:blockmatrixcr}
\mathfrak M=\mathfrak M_{\mathcal K}\left(F,\{a_j\}_{j=1}^m\right):=
\begin{pmatrix}
\mathcal A^+ & \mathcal B^- \\
\mathcal B^+ & -\mathcal A^-
\end{pmatrix},
\end{equation}
which has $2M$ rows and $2m+1$ columns. 

\par Now we are in a position to state our main result, which extends Theorem A from $H^1$ to $H^1_{\mathcal K}$. To keep on the safe side, we specify that the number $M:=\#\mathcal K$ is also allowed to be $0$; in this special case, we have $\mathcal K=\emptyset$, so that $H^1_{\mathcal K}$ reduces to $H^1$ and we are back to the classical situation. Our main concern is, however, the case where $M$ is a positive integer. 

\begin{thm}\label{thm:extrpoipuhonecr} Suppose that $f\in H^1_{\mathcal K}$ and $\|f\|_1=1$. Assume further that $f=IF$, where $I$ is inner and $F$ is outer. Then $f$ is an extreme point of $\text{\rm ball}(H^1_{\mathcal K})$ if and only if the following two conditions hold: 
\par{\rm (a)} $I$ is a finite Blaschke product whose degree, say $m$, does not exceed $M$.
\par{\rm (b)} The matrix $\mathfrak M=\mathfrak M_{\mathcal K}\left(F,\{a_j\}_{j=1}^m\right)$, built as above from $F$ and the zeros $\{a_j\}_{j=1}^m$ of $I$, has rank $2m$.
\end{thm}

\par To see a simple example, suppose that $\mathcal K$ consists of a single element, an integer $k(=k_1)$ with $k\ge2$. Thus, $\mathcal K=\{k\}$ and the subspace in question is 
$$H^1_{\{k\}}:=\{f\in H^1:\,\widehat f(k)=0\}.$$
Now let $F\in H^1$ be an outer function with $\|F\|_1=1$ and $\widehat F(k-1)=0$; then put $f(z):=zF(z)$. Clearly, $f\in H^1_{\{k\}}$ and $\|f\|_1=1$. Using Theorem \ref{thm:extrpoipuhonecr} with $M=m=1$, we verify (via a calculation, which we omit) that $f$ is an extreme point of $\text{\rm ball}\left(H^1_{\{k\}}\right)$ if and only if $|\widehat F(k-2)|\ne|\widehat F(k)|$. 

\par As regards possible applications of Theorem \ref{thm:extrpoipuhonecr}, one may recall first that Theorem A was crucial in describing the isometries of $H^1$; see \cite{dLRW} and \cite[Chapter 9]{Hof}. It is therefore conceivable that Theorem \ref{thm:extrpoipuhonecr} might serve a similar purpose in the $H^1_{\mathcal K}$ setting. 

\par We conclude this section by mentioning several types of subspaces in $H^1$, other than $H^1_{\mathcal K}$, where the geometry of the unit ball has been studied. This was done for shift-coinvariant subspaces \cite{DKal, DSib} and, more generally, for kernels of Toeplitz operators in $H^1$ \cite{DPAMS}. Also considered were spaces of polynomials of fixed degree, along with their Paley--Wiener type counterparts \cite{DMRL2000}, and quite recently, spaces of {\it lacunary} polynomials with prescribed spectral gaps \cite{DAdv2021}. This last-mentioned paper is especially close in spirit to our current topic. 

\par The rest of this note is devoted to proving Theorem \ref{thm:extrpoipuhonecr}. The bulk of the proof is deferred to Section 3 below, while Section 2 provides a couple of preliminary lemmas. The proofs are somewhat sketchy; full details and a more complete discussion can be found in \cite{DARX}. There, we also supplement Theorem \ref{thm:extrpoipuhonecr} with a result concerning the exposed points of $\text{\rm ball}(H^1_{\mathcal K})$. 

\section{Preliminaries} 

Two lemmas will be needed. When stating them, we write $L^\infty_\R$ for the set of real-valued functions in $L^\infty=L^\infty(\T)$. 

\begin{lem}\label{lem:charextrhonecr} Let $X$ be a subspace of $H^1$. Suppose that $f\in X$ is a function with $\|f\|_1=1$ whose canonical factorization is $f=IF$, with $I$ inner and $F$ outer. The following conditions are equivalent. 
\par{\rm (i)} $f$ is not an extreme point of $\text{\rm ball}(X)$. 
\par{\rm (ii)} There exists a function $G\in H^\infty$, other than a constant multiple of $I$, for which $G/I\in L^\infty_\R$ and $FG\in X$. 
\end{lem}

\begin{proof} We begin by restating condition (i). In fact, for $X$ as above, it is known (see \cite [Chapter V, Section 9]{Gam}) that a unit-norm function $f\in X$ is a non-extreme point of $\text{\rm ball}(X)$ if and only if there is a nonconstant function $h\in L^\infty_\R$ satisfying $fh\in X$. 
\par Now, if such an $h$ can be found, then $g:=fh$ is in $X$ and condition (ii) is fulfilled with $G:=Ih(=g/F)$. To check that this $G$ is in $H^\infty$, one may note that $Ih\in L^\infty$ and $g/F\in N^+$, where $N^+$ is the Smirnov class (see \cite[Chapter II]{G}). 
\par Conversely, if (ii) holds with a certain $G\in H^\infty$, then $h:=G/I$ is a nonconstant function in $L^\infty_\R$ and $fh(=FG)\in X$. 
\end{proof}

\par Before proceeding with the next result, we pause to introduce a certain class of polynomials that will be needed below. 

\par Given a nonnegative integer $N$ and a polynomial $p$, we say that $p$ is {\it $N$-symmetric} if $\ov z^Np(z)\in\R$ for all $z\in\T$. Equivalently, a polynomial $p$ is $N$-symmetric if (and only if) 
$$\widehat p(N-k)=\ov{\widehat p(N+k)}$$
for all $k\in\Z$; this accounts for the terminology. It follows that the general form of such a polynomial is 
\begin{equation}\label{eqn:gensym}
p(z)=\sum_{k=0}^{N-1}\left(\al_{N-k}-i\be_{N-k}\right)z^k+2\al_0z^N+
\sum_{k=N+1}^{2N}\left(\al_{k-N}+i\be_{k-N}\right)z^k,
\end{equation}
where $\al_0,\dots,\al_N$ and $\be_1,\dots,\be_N$ are real parameters. Arranging these into a vector 
\begin{equation}\label{eqn:coevec}
(\al_0,\al_1,\dots,\al_N,\be_1,\dots,\be_N)\in\R^{2N+1},
\end{equation}
which we call the {\it coefficient vector} of $p$, we arrive at a natural isomorphism between the space of $N$-symmetric polynomials and $\R^{2N+1}$. 

\begin{lem}\label{lem:genformrealcr} Given $N\in\Z_+$ and points $a_1,\dots,a_N\in\D$, let 
$$B(z):=\prod_{j=1}^N\f{z-a_j}{1-\ov a_jz}.$$
The general form of a function $\psi\in H^\infty$ satisfying $\psi/B\in L^\infty_\R$ is then $\psi=p\Phi$, where 
\begin{equation}\label{eqn:capphi}
\Phi(z):=\prod_{j=1}^N(1-\ov a_jz)^{-2}
\end{equation}
and $p$ is an $N$-symmetric polynomial. (If $N=0$, the products are taken to be $1$.)
\end{lem}

\begin{proof} If $\psi=p\Phi$, with $p$ an $N$-symmetric polynomial, then it is indeed true that the ratio $\psi/B$ is real-valued on $\T$ (and hence lies in $L^\infty_\R$). To see why, use the identity 
\begin{equation}\label{eqn:multident}
\psi/B=\left(\ov z^Np\right)\cdot\left(z^N\Phi/B\right)
\end{equation}
and the inequality $z^N\Phi/B\ge0$, both valid on $\T$. 
\par Conversely, suppose $\psi\in H^\infty$ is such that $\psi/B\in L^\infty_\R$. Using this last property in the form $\psi/B=\ov\psi/\ov B$, we infer that $\psi$ is orthogonal (in the Hardy space $H^2$) to the shift-invariant subspace $\th H^2$, where $\th:=zB^2$. In other words, $\psi$ belongs to the {\it star-invariant} (or {\it model}) subspace $H^2\ominus\th H^2$. Furthermore, because $\th$ is a finite Blaschke product, it follows (see, e.g., \cite{DSS} or \cite{N}) that $\psi$ is a rational function whose poles, counted with their multiplicities, are contained among those of $\th$ and which satisfies $\lim_{z\to\infty}\psi(z)/\th(z)=0$. This means that $\psi$ is expressible as $p\Phi$ for {\it some} polynomial $p$ of degree at most $2N$. Once this is known, we finally verify that $p$ is $N$-symmetric by invoking \eqref{eqn:multident} and the inequality stated next to it. 
\end{proof}

\section{Proof of Theorem \ref{thm:extrpoipuhonecr}}

We shall proceed in two steps. First, we prove the necessity of condition (a). Second, we show that condition (b) characterizes the extreme points among those unit-norm functions which obey (a). 

\par{\it Step 1.} Assuming that $I$, the inner factor of $f$, does not reduce to a finite Blaschke product of degree at most $M$ (so that (a) fails), we want to conclude that $f$ is not an extreme point of $\text{\rm ball}(H^1_{\mathcal K})$. By Lemma \ref{lem:charextrhonecr}, it suffices to construct a function $G\in H^\infty$, not a constant multiple of $I$, with the properties that
\begin{equation}\label{eqn:firprocr}
G/I\in L^\infty_\R
\end{equation}
and 
\begin{equation}\label{eqn:secprocr}
FG\in H^1_{\mathcal K}.
\end{equation}

\par We know from Frostman's theorem (see \cite[Chapter II]{G}) that there exists a point $w\in\D$ for which 
$$\ph:=\f{I-w}{1-\ov wI}$$
is a (finite or infinite) Blaschke product. Furthermore, our current assumption on $I$ guarantees that $\ph$ has at least $M+1$ zeros. Consequently, $\ph$ admits a factorization 
\begin{equation}\label{eqn:phiphiphicr}
\ph=\ph_1\ph_2, 
\end{equation}
where $\ph_1,\ph_2$ are Blaschke products and $\ph_1$ has precisely $M+1$ zeros. Setting $N:=M+1$, we therefore have
\begin{equation}\label{eqn:finblacr}
\ph_1(z)=\prod_{j=1}^N\f{z-a_j}{1-\ov a_jz}
\end{equation}
for some $a_1,\dots,a_N\in\D$. Next, we consider the function $g:=1-\ov wI\,(\in H^\infty)$ and observe that $I/\ph=g/\ov g$ a.e. on $\T$. When coupled with \eqref{eqn:phiphiphicr}, this yields
\begin{equation}\label{eqn:ieqprodcr}
I=\ph_1\ph_2g/\ov g.
\end{equation}

\par Our plan is to construct a function $G\in H^\infty$ satisfying \eqref{eqn:firprocr} and \eqref{eqn:secprocr} in the form 
\begin{equation}\label{eqn:intheformcr}
G=g^2p\Phi\ph_2,
\end{equation}
where $p$ is an $N$-symmetric polynomial and $\Phi$ is given by \eqref{eqn:capphi}, with the $a_j$'s coming from \eqref{eqn:finblacr}. In fact, any such $G$ belongs to $H^\infty$ and makes \eqref{eqn:firprocr} true. To verify the latter claim, combine \eqref{eqn:ieqprodcr} and \eqref{eqn:intheformcr} to find that 
$$G/I=G\ov I=|g|^2p\Phi\ov\ph_1$$
a.e. on $\T$; then apply Lemma \ref{lem:genformrealcr} with $B=\ph_1$ to deduce that the function $p\Phi\ov\ph_1$ (and hence also $|g|^2p\Phi\ov\ph_1$) is in $L^\infty_\R$. 

\par We also need to ensure \eqref{eqn:secprocr}, as well as the condition 
\begin{equation}\label{eqn:govineco}
G/I\ne\const,
\end{equation}
by choosing $p$ appropriately. To this end, we set $\mathcal F_0:=Fg^2\Phi\ph_2\,(\in H^1)$ and note that $FG=\mathcal F_0p$, so that \eqref{eqn:secprocr} boils down to requiring that the numbers 
$$\ga_j(p):=\widehat{(\mathcal F_0p)}(k_j),\qquad j=1,\dots,M,$$
be null. Now let $T$ be the linear map, defined on the space of all $N$-symmetric polynomials, that takes $p$ to the vector 
$$\left(\text{\rm Re}\,\ga_1(p),\,\text{\rm Im}\,\ga_1(p),\dots,
\text{\rm Re}\,\ga_M(p),\,\text{\rm Im}\,\ga_M(p)\right).$$
Identifying an $N$-symmetric polynomial $p$ with its coefficient vector (see Section 2 above), we may view $T$ as a linear operator from $\R^{2N+1}(=\R^{2M+3})$ to $\R^{2M}$. Its rank being obviously bounded by $2M$, we deduce from the rank-nullity theorem (see, e.g., \cite[p.\,63]{Axl}) that the kernel of $T$, to be denoted by $\mathfrak N_T$, satisfies $\text{\rm dim}\,\mathfrak N_T\ge3$. In particular, we can find two linearly independent $N$-symmetric polynomials, say $p_1$ and $p_2$, in $\mathfrak N_T$. When plugged into \eqref{eqn:intheformcr} in place of $p$, each of these makes \eqref{eqn:secprocr} true, while at least one of them ensures \eqref{eqn:govineco} as well. This completes the construction and proves the necessity of condition (a) in the theorem. 

\smallskip {\it Step 2.} From now on, we assume that (a) holds, so that $I$ is given by \eqref{eqn:finblaprodegmcr} with $0\le m\le M$ and $a_1,\dots,a_m\in\D$. In view of Lemma \ref{lem:charextrhonecr}, our function $f(=IF)$ will be an extreme point of $\text{\rm ball}(H^1_{\mathcal K})$ if and only if every $G\in H^\infty$ satisfying \eqref{eqn:firprocr} and \eqref{eqn:secprocr} is necessarily a constant multiple of $I$. Our purpose is therefore to prove that the latter condition is equivalent to (b). 

\par The functions $G\in H^\infty$ we should consider are those of the form 
\begin{equation}\label{eqn:geqpphicr}
G=p\Phi_0,
\end{equation}
where $p$ is an $m$-symmetric polynomial and 
$$\Phi_0(z):=\prod_{j=1}^m(1-\ov a_jz)^{-2}.$$
Indeed, Lemma \ref{lem:genformrealcr} (coupled with our current assumption on $I$) tells us that these are precisely the $G$'s that enjoy property \eqref{eqn:firprocr}. Now, if $F_0\in H^1$ is the function defined by \eqref{eqn:deffzerocr}, or equivalently by $F_0:=F\Phi_0$, then we have $FG=F_0p$ and condition \eqref{eqn:secprocr} takes the form 
\begin{equation}\label{eqn:saporana}
\widehat{(F_0p)}(k_j)=0,\qquad j=1,\dots,M.
\end{equation}
Rewriting these Fourier coefficients as convolutions and splitting each of the resulting equations into a real and imaginary part, we arrive at $2M$ real equations that recast \eqref{eqn:saporana} in terms of the coefficient vector 
\begin{equation}\label{eqn:coevecm}
(\al_0,\al_1,\dots,\al_m,\be_1,\dots,\be_m)
\end{equation}
of $p$. (It is understood that $p$ is related to \eqref{eqn:coevecm} as \eqref{eqn:gensym} is to \eqref{eqn:coevec}, but with $m$ in place of $N$.) Once these routine calculations are performed, we eventually rephrase \eqref{eqn:saporana}---and hence \eqref{eqn:secprocr}---by saying that the vector \eqref{eqn:coevecm} belongs to the subspace $\mathcal N:=\text{\rm ker}\,\mathfrak M$, the kernel of the linear operator $\mathfrak M:\R^{2m+1}\to\R^{2M}$ given by \eqref{eqn:blockmatrixcr}. 

\par This makes it easy to decide whether conditions \eqref{eqn:firprocr} and \eqref{eqn:secprocr} imply, for a function $G\in H^\infty$, that $G/I=\const$. Namely, this implication is valid if and only if $\dim\mathcal N=1$. To see why, one should observe first that $\mathcal N$ always contains a nonzero vector. Specifically, the $m$-symmetric polynomial 
$p_0:=I/\Phi_0$ solves \eqref{eqn:saporana}, so its coefficient vector is in $\mathcal N$. Now, if $\dim\mathcal N=1$, then any other $m$-symmetric polynomial $p$ solving \eqref{eqn:saporana} is given by $p=cp_0$ with some $c\in\R$; accordingly, the functions $G$ produced by \eqref{eqn:geqpphicr} are all of the form $G=cI$. Conversely, if $\dim\mathcal N>1$, then we can find an $m$-symmetric polynomial $p$ with $p/p_0\ne\const$ that makes \eqref{eqn:saporana} true, so the associated function $G(=p\Phi_0)$ satisfies \eqref{eqn:govineco}. 

\par To summarize, among the unit-norm functions $f=IF\in H^1_{\mathcal K}$ that obey (a), the extreme points of $\text{\rm ball}(H^1_{\mathcal K})$ are characterized by the property that $\mathcal N$, the kernel in $\R^{2m+1}$ of the linear map \eqref{eqn:blockmatrixcr}, has dimension $1$. Finally, another application of the rank-nullity theorem allows us to restate the latter condition as $\text{\rm rank}\,\mathfrak M=2m$, and we arrive at (b). The proof is complete.

\end{document}